\documentclass[3p,10pt]{elsarticle}

\usepackage{amsmath,amssymb,amsthm}

\journal{Finite Fields and Their Applications}

\newtheorem{theorem}{Theorem}[section]
\newtheorem{proposition}[theorem]{Proposition}
\newtheorem{lemma}[theorem]{Lemma}
\newtheorem{corollary}[theorem]{Corollary}
\theoremstyle{remark}
\newtheorem{remark}[theorem]{Remark}

\newcommand{\cC}{\mathcal{C}}
\newcommand{\cY}{\mathcal{Y}}
\newcommand{\cE}{\mathcal{E}}
\newcommand{\F}{\mathbb{F}}
\newcommand{\Tr}{\operatorname{Tr}}
\newcommand{\Ree}{\operatorname{Re}}

\begin{document}
\begin{frontmatter}

\title{Maximal and minimal curves of the form $y^3=x^{(q^2+1)/2}+x$}

\author[IFTM]{Guilherme Dias}
\ead{guilhermedias@iftm.edu.br}

\author[imecc]{Saeed Tafazolian\corref{cor1}}
\ead{saeed@unicamp.br}
\cortext[cor1]{Corresponding author.}

\address[IFTM]{Instituto Federal do Tri\^angulo Mineiro (IFTM), Campus Campina Verde, Fazenda Campo Belo, Rodovia BR-364, Km 153, Campina Verde, MG 38270-000, Brazil}

\address[imecc]{Departamento de Matem\'atica, Instituto de Matem\'atica, Estat\'istica e Computa\c{c}\~ao Cient\'ifica (IMECC), Universidade Estadual de Campinas (UNICAMP), Rua S\'ergio Buarque de Holanda, 651, Campinas, SP 13083-859, Brazil}

\begin{abstract}
Let $p\ge 5$ be a prime with $p\equiv -1\pmod 3$, let $q=p^r$, and consider
\[
\cC:\qquad y^3=x^{(q^2+1)/2}+x
\]
over $\F_{q^6}$. We prove the exact formula
\[
\#\cC(\F_{q^6})=q^6+1+(-1)^{r+1}(q^2-1)q^3.
\]
Since $g(\cC)=(q^2-1)/2$, the curve is maximal when $r$ is odd and minimal when $r$ is even. The proof uses a birational Kummer model and an explicit Jacobi-sum point count. A congruence together with Frobenius invariance reduces the relevant Jacobi sums to cubic Gauss sums, whose sign is determined from the Fermat cubic. In particular, the maximality of $y^3=x^{13}+x$ over $\F_{5^6}$ appears as the first case of an infinite family.
\end{abstract}


\end{frontmatter}

\section{Introduction}\label{sec:intro}

Let $\F_Q$ be a finite field of square cardinality. A smooth projective curve $\mathcal X/\F_Q$ of genus $g$ is maximal (respectively minimal) when equality holds in the Hasse--Weil bound with the plus (respectively minus) sign:
\[
\#\mathcal X(\F_Q)=Q+1\pm 2g\sqrt Q;
\]
see \cite{HirschfeldKorchmarosTorres,Stichtenoth}.

Curves of the form
\[
\cC(n,m):\qquad y^n=x^m+x
\]
form a natural family of superelliptic curves over finite fields. For $n=3$, Tafazolian and Torres \cite{TafazolianTorres2014} gave a complete maximality criterion over $\F_{p^2}$ when the square root of the field size is prime. Nevertheless, the curve
\[
y^3=x^{13}+x
\]
is maximal over $\F_{5^6}=\F_{125^2}$ \cite[Remark~3.8]{DiasTafazolian2025}; the prime-field criterion does not apply because $125$ is not prime.

The aim of this note is to explain this phenomenon uniformly. Let $p\ge 5$ be prime, $p\equiv -1\pmod 3$, let $r\ge 1$, and put $q=p^r$. Our main result is the following.

\begin{theorem}\label{thm:main}
The curve
\[
\cC:\qquad y^3=x^{(q^2+1)/2}+x
\]
has genus
\[
g=\frac{q^2-1}{2}
\]
and
\[
\#\cC(\F_{q^6})=q^6+1+(-1)^{r+1}(q^2-1)q^3.
\]
Consequently, $\cC$ is maximal over $\F_{q^6}$ if $r$ is odd and minimal if $r$ is even.
\end{theorem}

The proof is short once the curve is written in a suitable Kummer form. With
\[
d=\frac{q^2-1}{2},
\]
we obtain the birational model
\[
v^{3d}=t(1-t)^d.
\]
Its rational points are expressed through Jacobi sums. The congruence
\[
q^4\equiv d+1\pmod{3d}
\]
and Frobenius invariance of Gauss sums reduce every relevant term to a cubic Gauss sum. The remaining sign is read off from the maximal/minimal behavior of the Fermat cubic. The role of the condition $p\equiv -1\pmod 3$ is discussed in Remark~\ref{rem:p1mod3}.

We also place the genus $(q^2-1)/2$ in the known spectrum of genera arising from Hermitian quotients. Thus the novelty is not the occurrence of this genus itself, but rather the explicit infinite superelliptic family, the exact point-count formula valid for every $r\ge 1$, and the resulting parity-dependent transition between maximality and minimality. In particular, the construction gives a uniform explanation for a previously isolated maximal example while remaining entirely explicit at the level of equations and character sums.

\section{Character-sum preliminaries}\label{sec:prelim}

Let $Q=p^f$ and let $\varepsilon$ denote the trivial multiplicative character of $\F_Q$. Nontrivial multiplicative characters are extended to $\F_Q$ by $\chi(0)=0$. We use the standard additive character
\[
\zeta(u)=\exp\!\left(\frac{2\pi i}{p}\Tr_{\F_Q/\F_p}(u)\right),
\]
and write
\[
g(\chi)=\sum_{u\in\F_Q}\chi(u)\zeta(u),
\qquad
J(\chi_1,\chi_2)=\sum_{u\in\F_Q}\chi_1(u)\chi_2(1-u).
\]

We shall use the following standard facts \cite[Secs.~8.1, 8.3]{IrelandRosen}. If $\chi$ has exact order $n\mid(Q-1)$, then for $a\ne 0$,
\begin{equation}\label{eq:roots}
\#\{x\in\F_Q:x^n=a\}=\sum_{j=0}^{n-1}\chi^j(a),
\end{equation}
while for nontrivial $\chi_1,\chi_2,\chi_1\chi_2$,
\begin{equation}\label{eq:JG}
J(\chi_1,\chi_2)=\frac{g(\chi_1)g(\chi_2)}{g(\chi_1\chi_2)}.
\end{equation}
Moreover,
\[
|g(\chi)|=\sqrt Q\qquad(\chi\ne\varepsilon),
\]
and
\begin{equation}\label{eq:frob}
g(\chi^{p^k})=g(\chi)\qquad(k\ge 0).
\end{equation}
Indeed, \eqref{eq:frob} follows by applying Frobenius to the summation variable and using invariance of the trace.

We shall also use the following characterization of maximal and minimal Fermat curves.

\begin{theorem}[{\cite[Theorem~5]{minimalandmaximal}}]\label{thm:Fermat}
Let
\[
\mathcal F_N:\qquad X^N+Y^N=Z^N
\]
be the Fermat curve of degree $N$ over $\F_{p^{2s}}$. Then:
\begin{enumerate}
\item[\rm (i)] $\mathcal F_N$ is maximal over $\F_{p^{2s}}$ if and only if $N\mid p^s+1$;
\item[\rm (ii)] $\mathcal F_N$ is minimal over $\F_{p^{2s}}$ if and only if $s$ is even and there exists a divisor $e$ of $s/2$ such that $N\mid p^e+1$.
\end{enumerate}
\end{theorem}

\begin{lemma}\label{lem:cubic}
Let $Q=q^6=p^{6r}$ with $p\equiv -1\pmod 3$, and let $\lambda$ be a character of order $3$ on $\F_Q^*$. Then
\[
g(\lambda)=g(\lambda^2)=(-1)^{r+1}q^3.
\]
\end{lemma}

\begin{proof}
Since $p\equiv -1\pmod 3$, one has
\[
\lambda^p=\lambda^{-1}=\lambda^2.
\]
Hence \eqref{eq:frob} gives
\[
g(\lambda^2)=g(\lambda^p)=g(\lambda).
\]

Consider the Fermat cubic
\[
\cE:\qquad U^3+V^3=W^3.
\]
Apply Theorem~\ref{thm:Fermat} with $N=3$ and $s=3r$. If $r$ is odd, then $3\mid p^{3r}+1$, so $\cE$ is maximal over $\F_{p^{6r}}$. If $r$ is even, then $s=3r$ is even; taking $e=3$, we have $e\mid s/2$ and $3\mid p^3+1$, so $\cE$ is minimal over $\F_{p^{6r}}$.

With our extension of nontrivial characters at zero, the number of cube roots of $a\in\F_Q$ is
\[
1+\lambda(a)+\lambda^2(a).
\]
Thus the affine part of $\cE$ has
\[
\sum_{t\in\F_Q}
\bigl(1+\lambda(t)+\lambda^2(t)\bigr)
\bigl(1+\lambda(1-t)+\lambda^2(1-t)\bigr)
\]
points. Using
\[
J(\chi,\chi^{-1})=-\chi(-1)
\]
and $\lambda(-1)=1$, we get
\[
J(\lambda,\lambda^2)=J(\lambda^2,\lambda)=-1.
\]
Since $\cE$ has three $\F_Q$-rational points at infinity,
\[
\#\cE(\F_Q)=Q+1+J(\lambda,\lambda)+J(\lambda^2,\lambda^2).
\]
Complex conjugation sends $\lambda$ to $\lambda^2$, so
\[
J(\lambda^2,\lambda^2)=\overline{J(\lambda,\lambda)}.
\]
Hence
\[
\#\cE(\F_Q)=Q+1+2\Ree J(\lambda,\lambda),
\]
and therefore
\[
\Ree J(\lambda,\lambda)=(-1)^{r+1}q^3.
\]
On the other hand, \eqref{eq:JG} and $g(\lambda^2)=g(\lambda)$ give
\[
J(\lambda,\lambda)=\frac{g(\lambda)^2}{g(\lambda^2)}=g(\lambda).
\]
Since $|g(\lambda)|=\sqrt Q=q^3$, it follows that
\[
g(\lambda)=(-1)^{r+1}q^3,
\]
and the same equality holds for $\lambda^2$.
\end{proof}

\begin{remark}\label{rem:p1mod3}
The hypothesis $p\equiv -1\pmod 3$ enters essentially in Lemma~\ref{lem:cubic}: it implies $\lambda^p=\lambda^2$ and therefore $g(\lambda)=g(\lambda^2)$. If instead $p\equiv 1\pmod 3$, then $\lambda^p=\lambda$, so Frobenius no longer interchanges the two nontrivial cubic characters. The above argument therefore does not force the two cubic Gauss sums to have the same value. We do not address that case here.
\end{remark}

\section{Point count and proof of the main theorem}\label{sec:main}

From now on $Q=q^6$. Put
\[
m=\frac{q^2+1}{2},
\qquad
d=m-1=\frac{q^2-1}{2}.
\]
Since $q^2\equiv 1\pmod 3$ and $q$ is odd, we have $3\mid d$ and $3\nmid m$. Moreover $p\nmid d$. The polynomial
\[
f(x)=x^m+x=x(x^d+1)
\]
is separable: $f'(0)=1$, while if $x\ne 0$ and $f(x)=0$, then $x^d=-1$ and
\[
f'(x)=mx^d+1=1-m=-d\ne 0.
\]
Since $\gcd(3,m)=1$, the standard genus formula for superelliptic curves \cite[Proposition~6.3.1]{Stichtenoth} gives
\begin{equation}\label{eq:genus}
g(\cC)=\frac{(3-1)(m-1)}{2}=d=\frac{q^2-1}{2}.
\end{equation}
For completeness, the same value follows directly from Riemann--Hurwitz. The degree-three map
\[
x:\cC\longrightarrow \mathbb P^1
\]
is totally ramified above the $m$ distinct zeros of $f(x)$ and above the unique point at infinity, because $\gcd(3,m)=1$. Hence
\[
2g(\cC)-2=3(-2)+(m+1)(3-1)=2m-4,
\]
which again gives $g(\cC)=m-1=d$.

\begin{lemma}\label{lem:birat}
Over $\F_Q$, the curve $\cC$ is birational to
\[
\cY:\qquad v^{3d}=t(1-t)^d.
\]
\end{lemma}

\begin{proof}
Since
\[
Q-1=(q^2-1)(q^4+q^2+1)
\]
and $q^4+q^2+1\equiv 0\pmod 3$, we have
\[
6d=3(q^2-1)\mid Q-1.
\]
Choose $\beta\in\F_Q^*$ of order $6d$, so that $\beta^{3d}=-1$. Setting $t=-x^d$ and $v=\beta y$ gives
\[
v^{3d}=-(y^3)^d=-x^d(x^d+1)^d=t(1-t)^d.
\]
Conversely,
\[
y=\beta^{-1}v,
\qquad
x=\frac{y^3}{1-t},
\]
so the two function fields coincide.
\end{proof}

Fix a multiplicative character $\chi$ of exact order $3d$ on $\F_Q^*$.

\begin{proposition}\label{prop:count}
The smooth projective model of $\cY$ satisfies
\begin{equation}\label{eq:pointcount}
\#\cY(\F_Q)=Q+1+
\sum_{\substack{1\le j\le 3d-1\\3\nmid j}}
J(\chi^j,\chi^{dj}).
\end{equation}
\end{proposition}

\begin{proof}
For the Kummer extension
\[
v^{3d}=t(1-t)^d,
\]
the valuation of the right-hand side at $t=0$ is $1$, so the extension has degree $3d$. The only ramified places of $\F_Q(t)$ are $t=0$, $t=1$, and $t=\infty$, with corresponding valuations
\[
1,\qquad d,\qquad -(d+1).
\]
Thus
\[
\gcd(3d,1)=1,
\qquad
\gcd(3d,d)=d,
\qquad
\gcd(3d,d+1)=1.
\]
Hence there is a unique place above $t=0$ and a unique place above $t=\infty$.

We verify that the $d$ places above $t=1$ are rational. Since $q$ is odd,
\[
q^2\equiv 1\pmod 8,
\]
so $d=(q^2-1)/2$ is even. Put $u=t-1$. Then
\[
t(1-t)^d=u^d(1+u),
\]
and locally at $u=0$,
\[
\left(\frac{v^3}{u}\right)^d=1+u.
\]
After reduction modulo $u$, the residual equation is
\[
Z^d=1.
\]
Since $d\mid Q-1$ and $p\nmid d$, it has exactly $d$ distinct roots in $\F_Q$. Hence there are exactly $d$ $\F_Q$-rational places above $t=1$. Therefore the total contribution above $t=0,1,\infty$ is $d+2$.

For $t\in\F_Q\setminus\{0,1\}$, \eqref{eq:roots} gives
\[
\#\{v:v^{3d}=t(1-t)^d\}
=
\sum_{j=0}^{3d-1}\chi^j(t)\chi^{dj}(1-t).
\]
The term $j=0$ contributes $Q-2$. If $3\mid j$ and $j\ne 0$, then $\chi^{dj}=\varepsilon$ and the inner sum is $-1$; there are $d-1$ such indices. If $3\nmid j$, both $\chi^j$ and $\chi^{dj}$ are nontrivial, and the inner sum is
\[
J(\chi^j,\chi^{dj}).
\]
Indeed, the omission of $t=0,1$ does not alter the Jacobi sum, since our nontrivial characters are extended by zero at the origin. Combining the unramified fibers with the $d+2$ rational places lying above the branch points gives
\[
(Q-2)-(d-1)+(d+2)=Q+1,
\]
which proves \eqref{eq:pointcount}.
\end{proof}

\begin{lemma}\label{lem:reduce}
For $1\le j\le 3d-1$ with $3\nmid j$, one has
\[
J(\chi^j,\chi^{dj})=g(\chi^{dj}).
\]
\end{lemma}

\begin{proof}
Because $q^2=2d+1$,
\[
q^4-(d+1)=d(4d+3).
\]
Since $3\mid d$, it follows that
\begin{equation}\label{eq:congruence}
q^4\equiv d+1\pmod{3d}.
\end{equation}
For $3\nmid j$, the characters $\chi^j$, $\chi^{dj}$, and $\chi^{(d+1)j}$ are nontrivial. Indeed, $\chi^j$ is nontrivial because $1\le j\le 3d-1$, and $\chi^{dj}=(\chi^d)^j$ is nontrivial because $\chi^d$ has order $3$. Moreover,
\[
\gcd(d+1,3d)=1,
\]
so $\chi^{(d+1)j}$ is nontrivial whenever $\chi^j$ is nontrivial. Thus \eqref{eq:JG} gives
\[
J(\chi^j,\chi^{dj})
=
\frac{g(\chi^j)g(\chi^{dj})}{g(\chi^{(d+1)j})}.
\]
By \eqref{eq:congruence},
\[
\chi^{(d+1)j}=\chi^{q^4j}=(\chi^j)^{p^{4r}}.
\]
Therefore \eqref{eq:frob} gives
\[
g(\chi^{(d+1)j})=g(\chi^j),
\]
and hence
\[
J(\chi^j,\chi^{dj})=g(\chi^{dj}).
\]
\end{proof}

\begin{proof}[Proof of Theorem~\ref{thm:main}]
By Lemma~\ref{lem:birat}, the smooth projective models of $\cC$ and $\cY$ are isomorphic over $\F_Q$. Let
\[
\lambda=\chi^d.
\]
Since $\chi$ has order $3d$, the character $\lambda$ has order $3$. For every $j$ with $3\nmid j$,
\[
\chi^{dj}=\lambda^j
\]
is either $\lambda$ or $\lambda^2$. There are exactly $2d$ such values of $j$ in $\{1,\ldots,3d-1\}$. Proposition~\ref{prop:count}, Lemma~\ref{lem:reduce}, and Lemma~\ref{lem:cubic} therefore yield
\[
\#\cC(\F_Q)
=
Q+1+2d(-1)^{r+1}q^3
=
q^6+1+(-1)^{r+1}(q^2-1)q^3.
\]
Together with \eqref{eq:genus} and $\sqrt Q=q^3$, this is the upper Hasse--Weil bound for odd $r$ and the lower bound for even $r$.
\end{proof}

\begin{corollary}\label{cor:q5}
The curve
\[
y^3=x^{13}+x
\]
has genus $12$ and is maximal over $\F_{5^6}$; indeed,
\[
\#\cC(\F_{5^6})=5^6+1+24\cdot 5^3=18626.
\]
\end{corollary}

Thus Corollary~\ref{cor:q5} recovers \cite[Remark~3.8]{DiasTafazolian2025} as the case $p=5$, $r=1$ of Theorem~\ref{thm:main}.

\begin{remark}
Theorem~\ref{thm:main} does not conflict with the prime-field classification of \cite{TafazolianTorres2014}. Here the relevant square field is
\[
\F_{q^6}=\F_{(q^3)^2},
\]
whose square root has cardinality $q^3=p^{3r}$ and is never prime. Thus the family lies genuinely outside the hypotheses of that result. The parity of $r$ is exactly what controls the sign in the point count above.
\end{remark}

\section{The genus and Hermitian quotients}\label{sec:hermitian}

For odd $r$, Theorem~\ref{thm:main} gives an $\F_{q^6}$-maximal curve of genus
\[
g=\frac{q^2-1}{2}.
\]
This value of the genus is not new in the spectrum of maximal curves. Indeed, Abd\'on and Quoos \cite{AbdonQuoos} showed that suitable $p$-subgroups
\[
G\subseteq\operatorname{Aut}(\mathcal H_{q^3})
\]
give quotient curves with
\[
g(\mathcal H_{q^3}/G)
=
\frac{p^{3r-v}(p^{3r-w}-1)}{2},
\qquad
0\le v\le 3r,
\quad
0\le w\le 3r-1.
\]
Taking $v=3r$ and $w=r$ gives
\[
g(\mathcal H_{q^3}/G)
=
\frac{p^{2r}-1}{2}
=
\frac{q^2-1}{2}
=
g(\cC).
\]
Moreover, subfields of the Hermitian function field are maximal over the same constant field; see Garcia, Stichtenoth and Xing \cite{GarciaStichtenothXing}. Thus the genus occurring in Theorem~\ref{thm:main} was already known to occur for maximal curves over $\F_{q^6}$.

This comparison is only meant to locate the genus of $\cC$ within the known spectrum. The quotient produced by the Hermitian construction and the curve considered here arise from different descriptions, and equality of genera alone gives no isomorphism between them. Nor does it imply that $\cC$ is a subcover of the Hermitian curve. It therefore remains natural to ask whether, for odd $r$, the curve
\[
\cC:\qquad y^3=x^{(q^2+1)/2}+x
\]
is covered over $\F_{q^6}$ by the Hermitian curve
\[
\mathcal H_{q^3}:\qquad Y^{q^3+1}=X^{q^3}+X.
\]
We leave this question for future investigation.

\section*{Acknowledgments}

The second author was partially supported by CNPq grant no.~302774/2025-4 and FAEPEX grant no.~3485/25. Support from FAPESP grant no.~2024/00923-6 is also acknowledged.

\end{document}